\documentclass[12pt,final]{amsproc}
\usepackage{amssymb,amsmath,amsthm,latexsym, amscd, upref}
\usepackage{mathrsfs}
\usepackage{a4wide}
\usepackage[usenames,dvipsnames]{color}
\usepackage[utf8]{inputenc}
\usepackage[T1]{fontenc}
\usepackage{dsfont}
\usepackage{tikz}
\usepackage{float}
\usepackage[mathscr]{euscript}
\usepackage{mathtools}
\usepackage{hyperref}
\usepackage{color}
\usepackage[all]{xy}
\usepackage{mathptmx}

\newcommand{\co}{\mathrm{c}_0}
\newcommand{\conv}{\mathrm{conv}}
\newcommand{\vertiii}[1]{{\left\vert\kern-0.35ex\left\vert\kern-0.35ex\left\vert #1 \right\vert\kern-0.35ex\right\vert\kern-0.35ex\right\vert}}

\theoremstyle{plain}
\newtheorem{thm}{Theorem}[section]
\newtheorem{prop}[thm]{Proposition}
\newtheorem{lem}[thm]{Lemma}
\newtheorem{cor}[thm]{Corollary}
\newtheorem{prob}[thm]{Problem}
\theoremstyle{definition}
\newtheorem{dfn}[thm]{Definition}
\theoremstyle{remark}
\newtheorem{rem}[thm]{Remark}

\numberwithin{equation}{section}

\begin{document}

\title[Weak compactness and FPP for affine bi-Lipschitz maps]{Weak compactness and fixed point property for affine bi-Lipschitz maps}

\author[C. S.~Barroso]{Cleon S.~Barroso}
\address{Departamento de Matem\'atica, Universidade Federal do Cear\'a,
Campus do Pici s/n, 60455-360 Fortaleza, Brazil.}
\email{cleonbar@mat.ufc.br}

\author[V. Ferreira]{Valdir Ferreira}
\address{Centro de Ci\^encias e Tecnologia, Universidade Federal do Cariri,
Cidade Universit\'aria s/n, 63048-080 Juazeiro do Norte, Brazil.}
\email{valdir.ferreira@ufca.edu.br}

\thanks{The research of the first-named author was supported in part by the Brazilian Grant 232883/2014-9 within the scope of Science Without Borders CAPES program. It was also partially supported by FUNCAP/CNPq/PRONEX Grant 00068.01.00/15. }

\begin{abstract}
Let $X$ be a Banach space and let $C$ be a closed convex bounded subset of $X$. It is proved that $C$ is weakly compact if, and only if, $C$ has the {\it generic}  fixed point property ($\mathcal{G}$-FPP) for the class of $L$-bi-Lipschitz affine mappings for every $L>1$. It is also proved that if $X$ has Pe\l czy\'nski's property $(u)$, then either $C$ is weakly compact, contains an $\ell_1$-sequence or a $\co$-summing basic sequence. In this case, weak compactness of $C$ is equivalent to the $\mathcal{G}$-FPP for the strengthened class of affine mappings that are uniformly bi-Lipschitz. We also introduce a generalized form of property $(u)$, called {\it property $(\mathfrak{su})$}, and use it to prove that if $X$ has property $(\mathfrak{su})$ then either $C$ is weakly compact or contains a wide-$(s)$ sequence which is uniformly shift equivalent. In this case, weak compactness in such spaces can also be characterized in terms of the $\mathcal{G}$-FPP for affine uniformly bi-Lipschitz mappings. It is also proved that every Banach space with a spreading basis has property $(\mathfrak{su})$, thus property $(\mathfrak{su})$ is stronger than property $(u)$. These results yield a significant strengthening of an important theorem of Benavides, Jap\'on-Pineda and Prus published in 2004. 
\end{abstract}

\maketitle

\section{Introduction}\label{sec:Int}
Describing and understanding topological phenomena remains one of the most active topics in functional analysis. The problem of describing weak compactness has so far particularly been a topic of great interest. In this paper we are concerned with the problem of whether weak compactness can be interpreted by the fixed point property (FPP). Recall that a topological space $C$ is said to have the FPP for a class $\mathcal{M}$ of maps if every $f\in \mathcal{M}$ with $f(C)\subset C$ has a fixed point. This problem has been studied from a number of topological viewpoints by several authors, see e.g. \cite{Kl,Flo,LS,DM,BKR,BPP} and references therein. The analysis of this problem in the metric context often leads to a fruitful interplay between the geometry and structural aspects of Banach spaces. This is witnessed in several works where weak compactness constitutes the FPP for affine nonexpansive mappings. A nonexpansive mapping is in other words nothing else but a $1$-Lipschitz mapping on a subset of a Banach space. For example, Lennard and Nezir \cite{LN} proved that if a Banach space $X$ contains a basic sequence $(x_n)$ which is asymptotically isometric to the $\co$-summing basis, then its closed convex hull $\overline{\conv}\big( \{ x_n\}\big)$ fails to have the FPP for affine nonexpansive mappings. It turns out that in such cases $\overline{\conv}\big( \{ x_n\}\big)$ is typically a non-weakly compact set.

An interesting relaxation of the FPP is the generic-FPP ($\mathcal{G}$-FPP), a notion first proposed in \cite{BPP}.  For a convex subset $M$ of a topological vector space $X$, denote by $\mathcal{B}(M)$ the family of all nonempty bounded, closed convex subsets of $M$.

\begin{dfn}[\cite{BPP}]\label{def:1sec1} A nonempty set $C\in\mathcal{B}(X)$ is said to have the $\mathcal{G}$-FPP for a class $\mathcal{M}$ of mappings if whenever $K\in \mathcal{B}(C)$ then every $f\in \mathcal{M}$ satisfying $f(K)\subset K$ has a fixed point.
\end{dfn}

There is quite a lot known on $\mathcal{G}$-FPP when $X$ is a Banach space. For instance, Dowling, Lennard and Turett \cite{DLT1,DLT2} proved for the case in which $X$ is either $\co$, $L_1(0,1)$ or $\ell_1$ that weakly compact sets $C\in \mathcal{B}(X)$ correspond precisely to those having the $\mathcal{G}$-FPP for affine nonexpansive maps. In  2004 Benavides, Jap\'on-Pineda and Prus proved, among other important results, the following facts.

\begin{thm}[Benavides, Jap\'on Pineda and Prus \cite{BPP}]\label{thm:BPP} Let $X$ be a Banach space and $C\in \mathcal{B}(X)$. Then
\begin{enumerate}
\renewcommand{\labelenumi}{(\roman{enumi})}
\item $C$ is weakly compact if and only if $C$ has the $\mathcal{G}$-FPP for continuous affine maps. 
\item If $X$ is either $\co$ (equipped with its usual supremum norm) or  $J_p$ (the James space), then $C$ is weakly compact if and only if $C$ has the $\mathcal{G}$-FPP for uniformly Lipschitzian affine maps.
\item If $X$ is an $L$-embedded Banach space, then $C$ is weakly compact if and only if it has the $\mathcal{G}$-FPP for nonexpansive affine mappings.
\end{enumerate}
\end{thm}

It was further proved in \cite{DLT2} that $\textrm{c}$ shares with $\co$ the characterization of weak compactness in terms of the $\mathcal{G}$-FPP for nonexpansive affine mappings. We refer the reader to \cite{JPS} for other related results. It is also worth stressing that norm-continuous affine maps are in fact weakly continuous. Thus, as already pointed out in \cite{BPP}, one direction of the statements in Theorem \ref{thm:BPP} easily follows from Schauder-Tychonoff's fixed point theorem. Recall that a map $f\colon C\to X$ is said to be uniformly Lipschitz if
\[
\sup_{x\neq y\in C,\, p\in \mathbb{N}} \frac{ \| f^p(x)  - f^p(y)\|}{\| x - y\|}<\infty,
\]
where $f^p$ denotes the $p^{\textrm{th}}$ iteration of the mapping $f$. In other terms, $f$ is uniformly Lipschitz whenever there is a constant $L>0$ such that
\[
\| f^p(x) - f^p( y)\| \leq L \| x - y \|\quad\textrm{for all } x, y\in C\quad\textrm{and}\quad p\in \mathbb{N}.
\]  
Clearly nonexpansive maps are uniformly Lipschitz. Henceforth we shall say that a map $f\colon C\to X$ is {\it uniformly bi-Lipschitz with constant $L\geq 1$} if there exist constants $c_1, c_2>0$ with $c_1^{-1} c_2\leq L$ and such that  
\[
c_1\| x - y\| \leq \| f^p(x) - f^p( y)\| \leq c_2 \| x - y \|\quad\textrm{for all } x, y\in C\quad\textrm{and}\quad p\in \mathbb{N}.
\]  
In this case we also simply say that $f$ is a {\it uniformly $L$-bi-Lipschitz mapping}. If the above inequality holds only true for $p=1$, then we will simply say that $f$ is {\it $L$-bi-Lipschitz}. Note that an isometry is just an $1$-bi-Lipschitz mapping and that the closer $L$ is to $1$, the closer $f$ is to be an isometry. 

\medskip 
At first sight one may be tempted to characterize weak compactness in terms of $\mathcal{G}$-FPP for nonexpansive maps. This is not however generally true. Indeed, in 2008 P.-K. Lin \cite{Lin} equipped $\ell_1$ with the norm
\[
\vertiii{ x }_{\mathscr{L}}=\sup_{k\in\mathbb{N}}\frac{8^k}{1+8^k}\sum_{n=k}^\infty | x(n)|\quad\textrm{for}\quad x=(x(n))_{n=1}^\infty \in \ell_1,
\]
and proved that every $C\in \mathcal{B}\big((\ell_1, \vertiii{\cdot}_{\mathscr{L}})\big)$ has FPP for nonexpansive maps. Hence the unit ball $B_{(\ell_1, \vertiii{\cdot}_{\mathscr{L}})}$ has the $\mathcal{G}$-$FPP$ for affine nonexpansive maps, but of course fails to be weakly compact.  Another interesting example is highlighted by the following result from the recent literature, due to T. Gallagher, C. Lennard and R. Popescu:

\begin{thm}[\cite{GLP}] Let $c$ be the Banach space of convergent scalar sequences. Then there exists a non-weakly compact set $C\in \mathcal{B}\big( (c, \|\cdot\|_\infty)\big)$ with the FPP for nonexpansive mappings. 
\end{thm}

It is natural therefore to ask whether weak compactness describes $\mathcal{G}$-FPP for the class of uniformly Lipschitz affine mappings in arbitrary Banach spaces. Precisely, the main focus of this works is the following problem.

\begin{prob}\label{prob:1} Let $X$ be a Banach space and $C\in \mathcal{B}(X)$. Assume that $C$ is not weakly compact. Does there exist a set $K\in \mathcal{B}(C)$ and a uniformly Lipschitz affine mapping $f\colon K\to K$ that is fixed-point free? 
\end{prob}

Let us point out that a positive answer would lead to the following characterization of weak compactness: {\it A closed convex bounded subset $C$ of a Banach space $X$ is weakly compact if and only if $C$ has the $\mathcal{G}$-FPP for the class of affine uniformly Lipschitz mappings}. 

One way to try solving Problem \ref{prob:1} would be to obtain a wide-$(s)$ sequence which uniformly dominates all of its subsequences; that is, a basic sequence $(x_n)$ such that for some positive constants $d$ and $D$ and every increasing sequence of integers $(n_i)\subset \mathbb{N}$, the following inequalities hold for all $n\in \mathbb{N}$ and all choice of scalars $(a_i)_{i=1}^n$
\begin{equation}\label{eqn:1int}
d\left| \sum_{i=1}^n a_i\right| \leq \left\| \sum_{i=1}^n a_i x_{n_i}\right\| \leq D \left\| \sum_{i=1}^n a_i x_i\right\|.
\end{equation}

This would certainly obstruct the $\mathcal{G}$-FPP for uniformly Lipschitz affine maps. As one knows, subsymmetric or quasi-subsymmetric basis (in the sense of \cite[Corollary 2.7]{ABDS}) are examples of such sequences. However due to unconditionality aspects, they might not be so available since unconditional basic sequences may not exist at all \cite{GM}. Another possibility would be trying to get wide-$(s)$ sequences $(x_n)$ that dominate all of their right shift sequences $(x_{n+p})$, but uniformly on $p$. This typically happens when special structures are available; for example, those equivalent  to $\co$ or $\ell_1$ as well (cf. also \cite[Theorem 1]{DLT1}, \cite[Theorem 4.2]{BPP}, \cite{LN} and \cite[Proposition 2.5.14]{MN}). Such a possibility would, though, imply  that shift operators induced by $(x_n)$ would be continuous. It turns out that this might be notoriously difficult, or even generally impossible. One reason is that the class of Hereditarily Indecomposable spaces (spaces that have no decomposable subspaces, cf. \cite{GM}) do not admit shift-equivalent basic sequences, that is, sequences $(x_n)$ which are equivalent to its one right-shift $(x_{n+1})$. Moreover, the Banach space $G$ was constructed by Gowers in \cite{G} has an unconditional basis for which the right shift operator is not norm-bounded. 

All these facts apparently show how difficult would be any approach producing an affirmative answer for Problem \ref{prob:1} using shift like mappings. 
 
The first goal of this paper is to solve Problem \ref{prob:1} for the class of $L$-bi-Lipschitz affine maps, where $L$ can be approached to one as much as possible. Precisely, it will be proved that if  $C\in \mathcal{B}(X)$ is not weakly compact then it fails to have $\mathcal{G}$-FPP for the class of $L$-bi-Lipschitz affine maps for every $L>1$ (Theorem \ref{thm:M1}). 

Let us stress that the basic idea behind the proof of Theorem \ref{thm:M1} is to build a basic sequence $(x_n)\subset C$ which dominates the summing basis of $\co$  and yet is equivalent to some of its convex combinations. This will give rise to a fixed-point free $L$-bi-Lipschitz affine mapping $f$ leaving invariant a set $K\in \mathcal{B}(C)$. As we shall see, the set $K$ is precisely the closed convex hull of $(x_n)$. As regards the map $f$, it will be essentially taken as the sum of a diagonal operator and a weighted shift map with properly chosen coefficients. This yields a new construction in metric fixed point theory and can make more transparent the challenges behind Problem \ref{prob:1}. The proof that $f$ is $L$-bi-Lipschitz relies strongly on the classical Principle of Small Perturbation. We also point out that our approach differs from that in \cite{BPP} where, because of the special nature of the spaces considered there, bilateral and right-shift maps were successfully used. 

Our second goal is to provide an affirmative answer to Problem \ref{prob:1} in spaces with Pe\l czy\'nski's property $(u)$ (Theorem \ref{thm:6.1}). The proof uses a local version of a classical result of R. C. James proved for spaces with unconditional basis (cf. Lemma \ref{lem:KL2}). 

The third and last goal of this work is to introduce a strengthened form of the Pe\l czy\'nski's property $(u)$, called {\it property $(\mathfrak{su})$}, and establish the $\mathcal{G}$-FPP in Banach spaces with such a property (Theorem \ref{thm:M3}). 

The remainder of the paper is organized as follows. In Section 2 we will set up the notation and terminology adopted in this work. In Section 3 we slightly recover a few ideas behind clever constructions of fixed-point free maps under the lack of weak compactness. In this section we also gather a set of auxiliary results used throughout the paper. In Section 4 we state and prove our first main result (Theorem \ref{thm:M1}). Section 5 contains a local version of a result of R. C. James which describes the internal structure of bounded, closed convex sets in spaces with property $(u)$. In Section 6 we formally state and prove the second main result of this paper (Theorem \ref{thm:6.1}). In Section 7 we introduce a theoretical notion (Definition \ref{dfn:2sec7}) which is a kind of shiftsification of the property $(u)$, and use it to also provide a structural description of bounded, closed convex sets in spaces having such a property. Finally, we indicate how to use it to prove our third main result (Theorem \ref{thm:M3}). In Section 8 we conclude this work with a few additional remarks and questions. 



\section{Notation and basic terminology}\label{sec2:Notation}
Throughout this paper $X$ will denote a Banach space. The notation used here is standard and mostly follows \cite{AK,C, D2}. In particular, a sequence $(x_n)$ in $X$ is called a basic sequence if it is a Schauder basis for its closed linear span $[x_n]$. In this case $\mathcal{K}$ will stand for the basis constant of $(x_n)$. Further, we will also denote by $P_n$ and $R_n$ the natural basis projections given by 
\[
P_n x= \sum_{i=1}^n x^*_i(x) x_i\quad \textrm{ and } \quad R_nx= x - P_nx,\quad x\in [x_n]
\] 
where $\{ x^*_i\}_{i=1}^\infty$ are the biorthogonal functionals of $(x_n)$. Recall that $\mathscr{K}:=\sup_n\| P_n\|$. As in \cite[p. 3]{BPP} we also recall that 
\begin{equation}\label{eqn:Theta}
\inf\big\{ \| x - y\| \colon x\in [x_i]_{i=1}^n, \,\, \| x\| \geq a\, \,\, y\in [x_i]_{i=n+1}^\infty, \,\,\, n\in \mathbb{N}\big\}\geq \frac{a}{\mathscr{K}},
\end{equation}
for every $a>0$. By $\mathrm{c}_{00}$ we denote the vector space of sequences of real numbers which eventually vanish. Let us now recall a few well-known notions from the Banach space theory. 

\begin{dfn}\label{def:2sec2} Let $(x_n)\subset X$ and $(y_n)\subset Y$ be two sequences, where $X, Y$ are Banach spaces. The sequence $(x_n)$ is said to dominate the sequence $(y_n)$ if there exists a constant $L>0$ so that 
\[
\Big\| \sum_{n=1}^\infty a_n y_n \Big\| \leq L \Big\| \sum_{n=1}^\infty a_n x_n \Big\|, 
\]
for all sequence $(a_n)\in \mathrm{c}_{00}$.
\end{dfn}

Observe that when $(x_n)$ and $(y_n)$ are both basic sequences, to say that $(x_n)$  dominates $(y_n)$ is the same as to say that the map $x_n\mapsto y_n$ extends to a linear bounded map between $[x_n]$ and $[y_n]$. The sequences $(x_n)$ and $(y_n)$ are said to be equivalent (also called $L$-equivalent, with $L\geq 1$) and one writes $(x_n)\sim_L (y_n)$, if for any $(a_i)\in \mathrm{c}_{00}$ one has that
\[
\frac{1}{L} \Big\| \sum_{i=1}^\infty a_i x_i \Big\| \leq \Big\| \sum_{i=1}^\infty a_i y_i \Big\| \leq L \Big\| \sum_{i=1}^\infty a_i x_i \Big\|. 
\]
The {\it summing basis} of $\co$ is the sequence $(s_n)_n$ in $\co$ where for $n\in \mathbb{N}$, $s_n$ is defined by
\[
s_n=e_1 + e_2 + \dots + e_n,
\]
and $(e_n)$ being the canonical basis of $\co$. It is well known that the sequence $(s_n)_n$ defines a Schauder basis for $( \co, \| \cdot\|_\infty)$. A sequence $(x_n)$ in a Banach space $X$ is then said to be equivalent to the summing basis of $\co$ if 
\[
(x_n)\sim_L (s_n)\,\, \textrm{ for some } L\geq 1.
\]

\begin{dfn} A sequence $(x_n)$ in $X$ is called  {\it seminormalized} if 
\[
0<\inf_n\| x_n\|\leq \sup_n\| x_n\|<\infty. 
\] 
\end{dfn}

The following additional notions were introduced by H. Rosenthal \cite{Ro}.

\begin{dfn}\label{defi:RosWids} A seminormalized sequence $(x_n)$ in $X$ is called:
\begin{enumerate}
\item[(i)] A non-trivial weak Cauchy sequence if it is weak Cauchy and non-weakly convergent. 
\item[(ii)] A wide-$(s)$ sequence if $(x_n)$ is basic and dominates the summing basis of $\co$.
\item[(iii)] An $(s)$-sequence if $(x_n)$ is weak-Cauchy and a wide-$(s)$ sequence.
\end{enumerate}
\end{dfn}

Finally, recall that a sequence of non-zero elements  $(z_n)$ of $X$ is called a convex block basis of a given sequence $(x_n)\subset X$ if there exist integers $n_1<n_2<\dots$ and scalars $c_1, c_2,\dots$ so that
\begin{enumerate}
\item[(iv)] $c_i\geq 0$ for all $i$ and $\sum_{i=n_j+1}^{n_{j+1}} c_i=1$ for all $j$.\vskip .1cm
\item[(v)] $z_j=\sum_{i=n_j + 1}^{n_{j+1}}c_i x_i$ for all $j$.
\end{enumerate}



\section{Auxiliary results}\label{sec:3}
The construction of affine fixed-point free maps usually relies on maps  which are defined by taking suitable convex combinations of some basic sequence $(x_n)$ in $X$. For example, in  \cite{BPP} the following maps were considered in the proof of Theorem \ref{thm:BPP}:
\[
f_0\Big( \sum_{n=1}^\infty t_n x_n\Big) = \sum_{n=1}^\infty t_n x_{n+1},
\]
and
\[
f_1\Big( \sum_{n=1}^\infty t_n x_n \Big) = t_2 x_1 + \sum_{n=1}^\infty t_{ 2n-1} x_{2n+1}  + \sum_{n=2}^\infty t_{2n} x_{2n-2}.
\]
It is interesting to mention that, according to the terminology of \cite{BPP}, $f_0$ and $f_1$ are respectively a unilateral shift and a bilateral shift map. 

As another instance, the authors in \cite{DLT2} have described weak compactness in $\co$ in terms of the $\mathcal{G}$-FPP for nonexpansive maps by considering the map:
\[
f_2\Big( \sum_{n=1}^\infty t_n x_n\Big) = \sum_{n\in \mathbb{N}} \sum_{j\in \mathbb{N}} \frac{1}{2^j} t_n x_{j+n}.
\]

If $X$ has a well-behaved structure then convex combinations like these ones can be dominated by $(x_n)$ which, broadly speaking, reflects on the FPP of such maps. Indeed, as we have mentioned before, one can always describe weak-compactness in terms of the $\mathcal{G}$-$FPP$ for  uniformly Lipschitz affine mappings when wide-$(s)$ sequences fulfilling (\ref{eqn:1int}) can be found. This is not, however, an easy task. Despite that, as we shall see, wide-$(s)$ sequences and the Principle of Small Perturbation are the main tools used here to prove our first result. Let us conclude this section by gathering a few auxiliary results that will be used throughout the paper.   

\begin{prop}[Proposition 2.2 \cite{Ro}]\label{prop:Ros1} Let $(x_j)$ be a non-trivial weak-Cauchy sequence in a Banach space. Then $(x_j)$ has an $(s)$-subsequence (and therefore a wide-$(s)$ sequence). 
\end{prop}

\begin{thm}[Rosenthal $\ell_1$-theorem] Every bounded sequence in a Banach space has either a weak-Cauchy subsequence or a subsequence which is equivalent to the standard unit basis of $\ell_1$. 
\end{thm} 

The following existence result is also due to Rosenthal (\cite[Proposition 2]{Ro1}), the proof of which will be included here for reader's convenience

\begin{prop}\label{prop:selection} Let $X$ be a Banach space and $(y_n)$ be a seminormalized sequence in $X$. Assume that no subsequence of $(y_n)$ is weakly convergent. Then $(y_n)$ admits a wide-$(s)$ subsequence. 
\end{prop}

\begin{proof}
If $(y_n)$ has no weak-Cauchy subsequence, then $(y_n)$ has an $\ell_1$-subsequence $(x_n)$ by the Rosenthal $\ell_1$-theorem. It is easy to see in this case that $(x_n)$ is wide-$(s)$. If otherwise $(y_n)$ has a weak-Cauchy subsequence $(y_{n_k})$, then from our assumption we get that $(y_{n_k})$ is a non-trivial weak-Cauchy sequence. By Proposition \ref{prop:Ros1}, $(y_{n_k})$ has an $(s)$-subsequence $(x_n)$ which is in particular wide-$(s)$. This concludes the proof. 
\end{proof}

Finally, we recall the Principle of Small Perturbations \cite[p. 13]{AK}.

\begin{thm}\label{thm:PSP} Let $(x_n)_{n=1}^\infty $ be a basic sequence in a Banach space $X$ with basis constant $\mathscr{K}$. If $(z_n)_{n=1}^\infty$ is a sequence in $X$ such that 
\begin{equation}\label{eqn:1PSP}
2\mathscr{K} \sum_{n=1}^\infty \dfrac{ \| x_n - z_n\|}{ \| x_n\| }=\theta <1,
\end{equation}
then there is an invertible bounded linear operator $A\colon X\to X$ with $A(x_n)=z_n$ for all $n\in \mathbb{N}$ and such that $\| A\| \leq 1 + \theta$ and $\| A^{-1}\| \leq (1 - 
\theta)^{-1}$. In particular, $(z_n)$ is a basic sequence and 
\begin{equation}\label{eqn:2PSP}
(1 - \theta) \Bigg\| \sum_{n=1}^\infty t_n x_n\Bigg\| \leq \Bigg\| \sum_{n=1}^\infty t_n z_n\Bigg\|\leq (1 + \theta) \Bigg\| \sum_{n=1}^\infty t_n x_n\Bigg\|,
\end{equation}
whenever $\sum_{n=1}^\infty t_n x_n$ converges. 
\end{thm}



\section{The $\mathcal{G}$-FPP in arbitrary Banach spaces}\label{sec:4}
Our first main result reads as follows. 

\begin{thm}\label{thm:M1} Let $X$ be a Banach space and $C\in \mathcal{B}(X)$. Then $C$ is weakly compact if and only if $C$ has the $\mathcal{G}$-FPP for $L$-bi-Lipschitz affine mappings for every $L>1$. 
\end{thm}

\begin{proof}
As we have mentioned before, if $C$ is weakly compact then it has the $\mathcal{G}$-FPP for any class of norm-continuous affine maps. Thus only the converse direction needs to be proved. Assume then that $C$ is not weakly compact and fix any real number $L>1$. By Eberlein-\v{S}mulian's Theorem, we can find a sequence $(y_n)$ in $C$ with no weakly convergent subsequences. Let $(x_n)$ be the wide-$(s)$ subsequence of $(y_n)$ given by Proposition \ref{prop:selection}. In order to prove the failure of the $\mathcal{G}$-FPP we need to exhibit a set $K\in \mathcal{B}(C)$ and a fixed-point free $L$-bi-Lipschitz affine mapping $f\colon K\to K$. 

As regards the set $K$, we let $K=\overline{\conv}(\{ x_n\})$.  Before starting with the construction of $f$, we need to set up an useful formula for $K$.  

\paragraph{Claim 1:} $K=\Bigg\{ \displaystyle\sum_{n=1}^\infty t_n x_n  \,\colon\, \textrm{each } t_n\geq 0\, \textrm{ and }\, \sum_{n=1}^\infty t_n=1\Bigg\}$. 

\begin{proof}[Proof of Claim 1] Let 
\[
M= \Bigg\{ \sum_{n=1}^\infty t_n x_n  \,\colon\, \textrm{ each } t_n\geq 0 \textrm{ and } \sum_{n=1}^\infty t_n=1\Bigg\}. 
\]
Let $(s_n)$ denote the summing basis of $\co$. The fact that $M$ is closed can easily be deduced from the fact that the mapping $T\colon ( [ x_n], \| \cdot\|) \to (\co, \| \cdot\|_\infty)$ given by 
\[
T\Bigg( \sum_{n=1}^\infty t_n x_n\Bigg)= \sum_{n=1}^\infty t_n s_n 
\]
is continuous, since $(x_n)$ dominates the summing basis $(s_n)$. Note that 
\[
\overline{\conv}(s_n)=\Bigg\{ \sum_{n=1}^\infty t_n s_n \colon\,  \textrm{ each } t_n \geq 0\,\textrm{ and }\, \sum_{n=1}^\infty t_n=1\Bigg\}
\]
and $T^{-1}\big( \overline{\conv}(s_n)\big)= M$, as desired. Now, once $M$ is closed then it is clear that $M=K$. 
\end{proof}

With the set $K$ in hand, we proceed to construct the map $f$. Let $\mathscr{K}$ be the basis constant of $(x_n)$. Since $(x_n)$ is seminormalized  there are some reals $0< a< b$ such that $a\leq\| x_n\| \le b$ for every $n\in \mathbb{N}$. Pick any $\theta\in (0,1)$ so that $\frac{ 1+ \theta}{1 - \theta}\leq L$. Next choose a sequence of positive scalars $(\alpha_n)$ satisfying:
\[
\displaystyle\frac{ 4b\mathscr{K}}{a}\sum_{n=1}^\infty \alpha_n\leq \theta. 
\]
It is obvious that such numbers can be found. We then define $f\colon K \to K$ as follows: if $\sum_{n=1}^\infty t_n x_n\in K$ then
\[
f\Big( \sum_{n=1}^\infty t_n x_n \Big) = (1 - \alpha_1) t_1 x_1 + \sum_{n=2}^\infty \big( (1- \alpha_n)t_n + \alpha_{n-1}t_{n-1}\big) x_n.
\]
Clearly $f$ is an affine fixed point free self map of $K$. It remains to show that $f$ is $L$-bi-Lipschitz. In order to verify this, we let 
\[
z_n= (1 - \alpha_n) x_n  + \alpha_n x_{n+1},\quad n\in \mathbb{N}.
\]
Notice that $(z_n)$ is a non-trivial convex combination of $(x_n)$ for which one has that
\[
f(x) = \sum_{n=1}^\infty t_n z_n\quad \forall\,  x:=\sum_{n=1}^\infty t_n x_n \in K. 
\]

\paragraph{Claim 2:} $(z_n)$ fulfills assumption (\ref{eqn:1PSP}) in Theorem \ref{thm:PSP}.

\begin{proof}[Proof of Claim 2] Indeed, notice that
\[
\| x_n - z_n\| = \alpha_n \| x_n - x_{n+1}\| \leq \alpha_n 2b.
\]
Hence
\[
2\mathscr{K}\sum_{n=1}^\infty \frac{\| x_n - z_n \|}{\| x_n\|}\leq 2\mathscr{K}\sum_{n=1}^\infty \frac{2b\alpha_n}{a}=\frac{4b\mathscr{K}}{a}\sum_{n=1}^\infty \alpha_n\leq \theta<1.
\]
This establishes the claim.
\end{proof}
Therefore Theorem \ref{thm:PSP} implies $(z_n)$ is basic and is equivalent to $(x_n)$. In addition, condition (\ref{eqn:2PSP}) yields that 
\[
(1 - \theta)\| x - y\| \leq \| f(x) - f(y)\| \leq (1 + \theta)\| x + y\|
\]
for every $x, y\in K$ and $f$ is $L$-bi-Lipschitz. The proof of theorem is complete.
\end{proof}

\begin{rem} It is not difficult to check that given $L>1$ and $n\in \mathbb{N}$ we can find the previous fixed point free mapping in such a way that $f$ and all the iterates $f^i$ with $i\leq n$ are $L$-bi-Lipschitz. To see this it is enough to consider $\theta\in (0,1)$ such that $(\frac{1 +\theta}{1 - \theta})^n\leq L$. It is worth to mention also that the condition $L>1$ cannot be weakened to the case of isometries, since $\ell_1$ can be renormed so that its closed unit ball has the FPP for nonexpansive mappings \cite{Lin}.
\end{rem}

As a corollary we can give the following characterization of reflexivity in Banach spaces:

\begin{cor} A Banach space $X$ is reflexive if and only if its closed unit ball has the $\mathcal{G}$-FPP for $L$-bi-Lipschitz affine mappings for every $L>1$. 
\end{cor}



\section{Bounded, closed convex sets in spaces with property $(u)$}\label{sec:5}
Recognizing local structures in Banach spaces are relevant in the study of the metric fixed point theory. The main result of this section supplies a local version of a well-known result of R. C. James. It is concerned with the internal structure of bounded, closed convex sets in spaces with Pe\l czy\'nski's property $(u)$. 

\begin{dfn}[Pe\l czy\'nski]\label{def:PelU}  A Banach space $X$ is said to have Pe\l czy\'nski's property $(u)$ if for every weak Cauchy sequence $(x_n)$ in $X$, there exists a sequence $(y_n)\subset X$ satisfying the properties below:
\begin{enumerate}
\item $\sum_{n=1}^\infty y_n$ is weakly unconditionally Cauchy $(WUC)$ series, i.e
\[
\sum_{n=1}^\infty | x^* ( y_n) | <\infty \quad\textrm{ for all } x^*\in X^*.
\]
\item $(x_n - \sum_{i=1}^n y_i)_n$ converges weakly to zero. 
\end{enumerate}
\end{dfn}

\begin{rem} A few known facts are in order: Banach spaces with an unconditional basis have property $(u)$ (cf. \cite[Proposition 3.5.3]{AK}). Other examples of spaces satisfying the property $(u)$ can be found in \cite{GL} where, for instance, it is shown that $L$-embedded spaces enjoy this property.  
\end{rem}

The following proposition was first mentioned by Knaust and Odell in \cite[pp. 153--154]{KO}, we include the proof here for the sake of completeness. Given a sequence $(x_n)$ in $X$, recall that a sequence $(\mathfrak{X}_n)$ is called a {\it convex block sequence} of $(x_n)$ if there is a sequence of finite subsets of integers $(F_n)$ such that
\[
\max F_1< \min F_2\leq \max F_2< \min F_3<\dots < \max F_n< \min F_{n+1} < \dots
\]
together with sets of positive numbers $\{ \lambda^n_i\colon i\in F_n\}\subset [0,1]$ satisfying $\sum_{i\in F_n} \lambda^n_i=1$ and $\mathfrak{X}_n=\sum_{i\in F_n} \lambda^n_i x_i$. 

\begin{prop}\label{prop:KO} Let $X$ be a Banach space with property $(u)$. Then every non-trivial weak Cauchy sequence in $X$ has a convex block subsequence equivalent to the summing basis of $\co$. 
\end{prop}

\begin{proof} Let $(x_n)$ be a non-trivial weak Cauchy sequence in $X$. By Proposition \ref{prop:Ros1}, $(x_n)$ has a subsequence $(x_{n_k})$ which is wide-$(s)$. In particular, $(x_{n_k})$ is basic and dominates the summing basis of $\co$ (Definition \ref{defi:RosWids}). Thus there exists a constant $c_1>0$ so that
\begin{equation}\label{eqn:1propKO}
c_1\sup_{1\leq k\leq m}\Bigg| \sum_{i=k}^m a_i \Bigg| \leq \Bigg\| \sum_{i=1}^m a_i x_{n_i}\Bigg\|,
\end{equation}
for every sequence of scalars $(a_i)_{i=1}^m\subset \mathbb{R}$. Let $(y_n)$ be a sequence in $X$ satisfying properties (1) and (2) in Definition \ref{def:PelU}. Note that such sequence exists because $X$ has property $(u)$. For $n\in \mathbb{N}$, define 
\[
z_n= \sum_{i=1}^n y_i. 
\]
Then $(x_{n_k}  - z_{n_k})$ is weakly null, by (2)--Definition \ref{def:PelU}. By Mazur's Theorem (\cite[p. 344]{AK}) there is a convex block sequence of $(x_{n_k} - z_{n_k})$ which converges in norm to zero. We then can find a convex block sequence $(\mathfrak{X}_n)$ of $(x_{n_k})$ and a convex block sequence $(\mathfrak{Z}_n)$ of $(z_{n_k})$ such that 
\begin{equation}\label{eqn:2propKO}
\sum_{i=1}^\infty \big\| \mathfrak{X}_i - \mathfrak{Z}_i\big\| \leq \frac{c_1}{4\mathscr{K}},
\end{equation}
where $\mathscr{K}$ is the basis constant of $(x_{n_k})$. Using (\ref{eqn:1propKO}) one can easily verify that 
\begin{equation}\label{eqn:3propKO}
c_1\sup_{1\leq k\leq m}\Bigg| \sum_{i=k}^m a_i \Bigg| \leq \Bigg\| \sum_{i=1}^m a_i \mathfrak{X}_i\Bigg\|,
\end{equation}
for all sequence of scalars $(a_i)_{i=1}^m\subset \mathbb{R}$. It is clear that $\mathscr{K}$ is also the basis constant of $(\mathfrak{X}_i)$. Combining this fact with (\ref{eqn:2propKO}) one can show that
\begin{equation}\label{eqn:4propKO}
\Bigg\| \sum_{i=1}^m a_i \mathfrak{X}_i\Bigg\| \leq 2  \Bigg\| \sum_{i=1}^m a_i\mathfrak{Z}_i\Bigg\|,
\end{equation}
for all scalars $(a_i)_{i=1}^m$. Indeed, note that (\ref{eqn:3propKO}) implies $c_1\leq \| \mathfrak{X}_i\|$ for all $i\geq 1$ and, consequently, we have
\[
\begin{split}
\Bigg\| \sum_{i=1}^m a_i \mathfrak{X}_i\Bigg\| &\leq \max_{1\leq i\leq m}| a_i| \sum_{i=1}^m\big\| \mathfrak{X}_i - \mathfrak{Z}_i\big\| + \Bigg\| \sum_{i=1}^m a_i\mathfrak{Z}_i\Bigg\|\\
&\leq \frac{ c_1}{4\mathscr{K}}\max_{1\leq i\leq m}| a_i| +  \Bigg\| \sum_{i=1}^m a_i\mathfrak{Z}_i\Bigg\|\\
&\leq \frac{1}{2} \Bigg\| \sum_{i=1}^m a_i \mathfrak{X}_i\Bigg\| + \Bigg\| \sum_{i=1}^m a_i\mathfrak{Z}_i\Bigg\|.
\end{split}
\]
We now claim that $(\mathfrak{Z}_i)$ is dominated by the summing basis of $\co$. To see this, write $\mathfrak{Z}_i= \sum_{k\in F_i}\lambda^i_k z_{n_k}$ where $(F_i)_i$ is an increasing sequence of block of natural numbers and $(\lambda^i_k)_{k\in F_i}$ are non-negative numbers satisfying $\sum_{k\in F_i}\lambda^i_k=1$ for all $i\geq 1$. Next fix any sequence of scalars $(a_i)_{i=1}^m\subset \mathbb{R}$ and pick a functional $x^*\in B(X^*)$ so that $\| \sum_{i=1}^m a_i \mathfrak{Z}_i\|= x^*\big( \sum_{i=1}^m a_i \mathfrak{Z}_i\big)$. Now choose for each $1\leq i\leq m$ an index $\kappa_i\in F_i$ such that $a_i x^*(z_{n_{\kappa_i}})= \max\{ a_i x^*( z_{n_k})\colon k\in F_i\}$. Then
\[
\begin{split}
\Bigg\| \sum_{i=1}^m a_i \mathfrak{Z}_i\Bigg\|&=x^*\Bigg( \sum_{i=1}^m a_i \mathfrak{Z}_i\Bigg) =  \sum_{i=1}^m \sum_{k\in F_i} \lambda^i_k  a_i x^*( z_{n_k})\\
& \leq \sum_{i=1}^m a_i x^*( z_{n_{\kappa_i}})\leq \Bigg\|\sum_{i=1}^m a_i z_{n_{\kappa_i}}\Bigg\|.
\end{split}
\]
From (1)--Definition \ref{def:PelU} we know that $\sum_{n=1}^\infty y_n$ is (WUC). So, we can apply Theorem 6 in \cite[p. 44, implication 1. $\implies$ 2.]{D2} to conclude that there is $c_2>0$ such that 
\[
\Bigg\| \sum_{i=1}^m t_i y_i\Bigg\| \leq c_2 \sup_{1\leq i\leq m}| t_i|
\]
for all scalars $(t_i)_{i=1}^m\subset \mathbb{R}$. From this inequality and the definition of $(z_n)$ it is easy to deduce
\[
\begin{split}
 \Bigg\|\sum_{i=1}^m a_i z_{n_{k_i}}\Bigg\| &= \Big\| a_1 z_{n_{\kappa_1}} + \dots + a_m z_{n_{\kappa_m}}\Big\|\\
 &=\Bigg\| \Big(\sum_{i=1}^m a_i \Big)y_1 + \dots + \Big(\sum_{i=1}^m a_i\Big)y_{n_{\kappa_1}} + \Big( \sum_{i=2}^m a_i\Big) y_{n_{\kappa_1 +1}} + \dots \\
 &\hskip 2cm  \dots + \Big( \sum_{i=2}^m a_i\Big) y_{n_{\kappa_2}} + \dots + (a_{m-1} + a_m) y_{n_{\kappa_{m-1}}} + a_m y_{n_{\kappa_m}}\Bigg\|\\
 &\leq c_2 \sup_{1\leq k\leq m}\Bigg| \sum_{i=k}^m a_i\Bigg|.
 \end{split}
\]
We therefore obtain that 
\[
\Bigg\| \sum_{i=1}^m a_i \mathfrak{Z}_i\Bigg\|\leq  \Bigg\|\sum_{i=1}^m a_i z_{n_{k_i}}\Bigg\|\leq c_2 \sup_{1\leq k\leq m}\Bigg| \sum_{i=k}^m a_i\Bigg|.
\]
This combined with (\ref{eqn:3propKO}) and (\ref{eqn:4propKO}) yields
\[
c_1\sup_{1\leq k\leq m}\Bigg| \sum_{i=k}^m a_i \Bigg| \leq \Bigg\| \sum_{i=1}^m a_i \mathfrak{X}_i\Bigg\|\leq 2c_2 \sup_{1\leq k\leq m}\Bigg| \sum_{i=k}^m a_i\Bigg|,
\]
for all scalars $(a_i)_{i=1}^m \subset \mathbb{R}$ and hence $(\mathfrak{X}_i)$ is a convex block subsequence of $(x_n)$ which is equivalent to the summing basis of $\co$. The proof of proposition is complete. 
\end{proof}

We are now ready to prove the main result of this section. 

\begin{lem}\label{lem:KL2} Let $X$ be a Banach space with the property $(u)$ and $C\in\mathcal{B}(X)$. Then either $C$ is weakly compact, $C$ contains an $\ell_1$-sequence or $C$ contains a $\co$-summing basic sequence. 
\end{lem}

\begin{proof} If $C$ is weakly compact, $C$ cannot contain neither a $\ell_1$-basic sequence nor a $\co$-summing basic sequence, due to the fact that these sequences fail to contain weak convergent subsequences. Assume now that $C$ is not weakly compact. Then it contains either an $\ell_1$-sequence or not. If so, the result follows. Otherwise, $C$ must contain a $\co$-summing basic sequence. Indeed, let $(x_n)\subset C$ be a weak-Cauchy sequence without weak convergent subsequences. This is possible thanks to Eberlein-\v{S}mulian's theorem and Rosenthal's $\ell_1$-theorem, as well. If $X$ has the property $(u)$, then so does the space $[(x_n)_n]$ (see \cite{Pel} (cf. also \cite[Proposition 3.5.2]{AK}).  By Proposition \ref{prop:KO} we can therefore deduce that $(x_n)$ has a convex block subsequence which is equivalent to the summing basis of $\co$. This concludes the proof. 
\end{proof}

\begin{rem} Notice that in the above result if $C$ is not weakly compact then it contains either an $\ell_1$-sequence or a $\co$-summing basic sequence. It is worth to mention that this is not an exclusive dichotomy. Indeed, consider $X=\ell_1\oplus \co$ which is a Banach space with an unconditional basis. Next denote by $\{ e_n\}$ and $\{ s_n\}$ the standard unit basis of $\ell_1$ and the summing basis of $\co$, respectively. Then any closed convex subset $C$ of $X$ containing both the sequences $\{ (e_n, 0)\}$ and $\{ (0, s_n)\}$ is an example of a non-weakly compact closed convex set containing both $\ell_1$ and $\co$-summing basic sequences. 
\end{rem}



\section{The $\mathcal{G}$-FPP in spaces with Pe\l czy\'nski's property $(u)$}\label{sec:6}
In this section we not only give an affirmative answer for Problem \ref{prob:1} in spaces with the property $(u)$, but we also provide an improvement to \cite[Theorem 3.2-(c)]{BPP}. 

Let $K$ be a nonempty convex subset of a Banach space $X$ and $f\colon K\to K$ an affine mapping. Following \cite[p. 9]{BPP} we define
\[
\theta(f)= \inf\big\{ \liminf_{n\to \infty}\| x - f^n(y)\| \colon x, y\in K\big\}.
\]
More precisely, we obtain the following result. 

\begin{thm}\label{thm:6.1} Let $X$ be a Banach space with the property $(u)$. Then $C\in \mathcal{B}(X)$ is weakly compact if and only if $C$ has the $\mathcal{G}$-FPP for the class of uniformly bi-Lipschitzian affine maps $f$ such that $\theta(f)=0$. 
\end{thm}

\begin{proof} It suffices to prove the converse implication. Assume that $C$ is not weakly compact. By Lemma \ref{lem:KL2} either $C$ contains an $\ell_1$-basic sequence or it contains a $\co$-summing basic sequence. In either case, $C$ contains a wide-$(s)$ sequence $(x_n)$ so that $(x_{n+p})$ is equivalent to $(x_n)$, but uniformly on $p\in \mathbb{N}$. Furthermore, passing to a subsequence if necessary, we may assume (cf. \cite[Fact 2.1]{BPP}) that there is a functional $\varphi\in X^*$ so that 
\[
\gamma=\inf\big\{ \varphi(x_n)\colon n\in \mathbb{N}\big\}>0.
\]
Hence for $K=\overline{\conv}\big( \{ x_n\}\big)$ the map $f\colon K\to K$  given by 
\[
f( x) = \sum_{n=1}^\infty t_n x_{n+1}\quad \textrm{for } \,x= \sum_{n=1}^\infty t_n x_n\in K,
\]
is affine, fixed point free and uniformly bi-Lipschitz. Therefore it remains only to show that $\theta(f)>0$. Notice that 
\[
K=\Bigg\{\sum_{n=1}^\infty t_n x_n\colon \text{each }\, t_n\geq 0\text{ and }\sum_{n=1}^\infty t_n=1\Bigg\}.
\]
Let $x= \sum_{n=1}^\infty t_n x_n$ and $y= \sum_{n=1}^\infty s_n x_n$ belong to $K$. Put $\beta= \gamma/\| \varphi\|$ and fix $0< \varepsilon< \beta/(1 + 2\mathscr{K})$ where $\mathscr{K}$ denotes the basis constant of $(x_n)$. Next find $m$ large enough so that $\|R_mx\|< \varepsilon$. Since $f$ is a right-shift it is easy to see that $\| P_m f^n(y)\|< \varepsilon$ for $n$ sufficiently large. By (\ref{eqn:Theta}) we have
\[
\begin{split}
\| x - f^n(y)\| &\geq \| P_m x - R_m f^n(y)\| - \| R_m x\| - \| P_m f^n y\| \geq \frac{ \| P_m x\|}{\mathscr{K}} -2\varepsilon\\
&\geq \frac{ \| x \| - \varepsilon}{\mathscr{K}} - 2\varepsilon \geq \frac{1}{\mathscr{K}}\big( \frac{\gamma}{\| \varphi\|} - \varepsilon) - 2\varepsilon = \frac{ \beta - \varepsilon( 1 + 2\mathscr{K})}{\mathscr{K}}.
\end{split}
\]
This shows that $\theta(f)>0$ and concludes the proof of the theorem. 
\end{proof}

We know that in $\co$ the standard unit basis $(e_i)_{i=1}^\infty$ is unconditional. In particular $\co$ has property $(u)$. Thus an immediate consequence of Theorem \ref{thm:6.1} is the following 

\begin{cor}[Theorem 3.2-(c), \cite{BPP}] Let $C\subset \co$ be a closed convex bounded set. If $C$ is not weakly compact, then there are a closed convex subset $K\subset C$ and an affine uniformly Lipschitzian mapping $f\colon K\to K$ such that $\theta(f)>0$. 
\end{cor}

Another immediate corollary of Theorem \ref{thm:6.1} is

\begin{cor} Let $X$ be a Banach space. Assume that $X$ is either $L$-embedded or has the hereditary DP-property. Then $C\in \mathcal{B}(X)$ is weakly compact if and only if it has the $\mathcal{G}$-FPP for the class of uniformly bi-Lipschitz affine mappings. 
\end{cor}

\begin{rem} Recall \cite{D} a Banach space $X$ is said to have the DP-property (i.e., Dunford-Pettis property) if for every pair of weakly null sequences $(x_n)\subset X$ and $(x^*_n)\subset X^*$ one has $\lim_{n\to \infty} \langle x_n , x^*_n\rangle=0$. Further, $X$ is said to have the hereditarily DP-property if all of its closed subspaces have the DP-property. It is also known (cf. proof of \cite[Theorem 2.1]{KO}) that spaces with hereditary DP-property have property $(u)$. 
\end{rem}

\begin{rem}\label{rem:23} The classical James' space $\mathrm{J}_2$ is a standard example of a space which fails property $(u)$. If $(e_i)$ denotes the canonical basis of $\mathrm{J}_2$, then recall the summing basis of  $\mathrm{J}_2$ is the sequence $(u_n)$ given by $u_n=\sum_{i=1}^n e_i$ ($n\in \mathbb{N}$). It is well known that 
\[
\Bigg\| \sum_{i=1}^\infty a_n u_n \Bigg\|_{\mathrm{J}_2}=\sup\Bigg\{ \Big( \sum_{k=1}^n\Big| \sum_{i=p_{k-1}}^{p_k -1} a_i \Big|^2\Big)^{1/2}\, \colon n\in \mathbb{N},\, p_0< \dots < p_n\Bigg\}.
\]
Furthermore, $(u_n)$ is a boundedly complete conditional spreading basis for $\mathrm{J}_2$ (e.g., see \cite[p. 1207]{AMS}). Recall a sequence in a Banach space is called {\it spreading} if it is equivalent to all of its subsequences. Similarly one can define the James's space $\rm{J}_p$ ($1<p<\infty)$ as being the jamesification of $\ell_p$ (see \cite[p. 5]{BPP}). Theorem \ref{thm:6.1} gives us a nice description of weak compactness in Banach spaces with property $(u)$. It would be thus interesting to know whether a similar result holds for $\mathrm{J}_p$. A positive answer would improve \cite[Theorem 3.2-(b)]{BPP}. In the next section we will discuss this issue in more detail. 
\end{rem}



\section{Banach spaces with property $(\mathfrak{su})$ and the $\mathcal{G}$-FPP}\label{sec:7}
In this section we shall introduce a generalization of the Pe\l czy\'nski's property $(u)$ and use it to give an affirmative answer for the question raised in Remark \ref{rem:23}. As a motivation we recall a notion of convex block homogeneous sequences introduced and studied recently by S. Argyros, P. Motakis and B. Sari \cite[Definition 3.1]{AMS}. 

\begin{dfn} Let $X$ be a Banach space and $(x_n)$ be a basic sequence in $X$.
\begin{itemize}
\item[(i)] If $(x_n)$ is equivalent to all of its convex block sequences, then $(x_n)$ is called convex block homogeneous. 
\item[(ii)] If $(x_n)$ is isometrically equivalent to all of its convex block sequences, then $(x_n)$ is called $1$-convex block homogeneous. 
\end{itemize}
\end{dfn}
These properties were heavily used in \cite{AMS} to decompose conditional spreading bases into two well behaved parts, one of which being unconditional and the other convex block homogeneous \cite[Theorem 4.1 and Remark 4.2]{AMS}. They were also used to study several structural properties of such spaces. There, it is seen that the unit vector basis of $\ell_1$, the summing basis of $\co$ and the boundedly complete basis of James's space $\mathrm{J}_2$ as well, are among the simplest examples of convex block homogeneous bases. 

The notion we are concerned with is a theoretical notion, which provides a kind of {\it shiftsification} of the notions of convex block homogeneous sequences and Pe\l czy\'nski's property $(u)$. 

\begin{dfn}[Convex Block Shiftable Sequences]\label{def:cbss} A basic sequence $(x_n)$ in a Banach space $X$ is said to be {\it convex block shiftable} if there exist an $L>0$ and a subsequence $(y_n)$ of $(x_n)$ such that every convex block sequence $(w_n)$ of $(y_n)$ is {\it uniformly $L$-shift equivalent}, i.e.  $(w_{n+p}) \sim_L (w_n)$ for every $p\in \mathbb{N}$. 
\end{dfn}

\begin{dfn}[Property $(\mathfrak{su})$]\label{dfn:2sec7} A Banach space $X$ is said to have {\it property $(\mathfrak{su})$}, if for every non-trivial weak Cauchy sequence $(x_n)$ in $X$ there exists a convex block shiftable sequence $(w_n)$ such that $(x_n - w_n)$ is weakly null. 
\end{dfn}

Now we show that property $(\mathfrak{su})$ constitutes a stronger reformulation of property $(u)$. 

\begin{prop}\label{prop:1sec7} Pe\l czy\'nski's property $(u)$ implies property $(\mathfrak{su})$. 
\end{prop}

\begin{proof} Let $X$ be a Banach space with property $(u)$ and $(x_n)$ a non-trivial weak Cauchy sequence in $X$. By Proposition \ref{prop:KO} there exists a convex block sequence $(\mathfrak{X}_n)$ of $(x_n)$ which is equivalent to the summing basis of $\co$. It is not difficult to verify that all convex block sequences of $(\mathfrak{X}_n)$ are also equivalent to the summing basis of $\co$. In particular $(\mathfrak{X}_n)$ is convex block shiftable. Finally, as it is easy to verify that $(x_n - \mathfrak{X}_n)$ is weakly null, the required property follows from Definition \ref{dfn:2sec7}. 
\end{proof}

\begin{prop}\label{prop:2sec7} If a Banach space $X$ has property $(\mathfrak{su})$ then every closed subspace $Y$ of $X$ also has property $(\mathfrak{su})$. 
\end{prop}

\begin{proof} The proof uses similar arguments as those in \cite[Proposition 3.5.2]{AK}, and we will only abbreviate it. Let $(y_n)$ be a non-trivial weak Cauchy sequence in $Y$. Then, up to a subsequence, by applying Mazur's lemma and using the assumption that $X$ has property $(\mathfrak{su})$, we can obtain a convex block sequence $(\mathfrak{Y}_k)$ of $(y_n)$ that is convex block shiftable and is such that $(y_k - \mathfrak{Y}_k)$ converges weakly to $0$. 
\end{proof}

\begin{thm}\label{thm:1sec7} Every Banach space with a spreading basis has property $(\mathfrak{su})$. 
\end{thm}

\begin{proof} Let $X$ be a Banach space and $(e_n)$ a spreading basis for $X$. By Proposition \ref{prop:1sec7} we can assume that $(e_n)$ is conditional.  We may also assume (after passing to an equivalent norm) that $(e_n)$ is $1$-spreading, i.e. $(e_n)$ is isometrically equivalent to its subsequences. Define $u_n:= e_{2n} - e_{2n-1}$, $n\in \mathbb{N}$. In \cite{AMS} the sequence $(u_n)$ is called the {\it subsymmetric skipped difference} of $(e_n)$. From \cite[Theorem 4.1]{AMS} (see also \cite[Remark 4.2]{AMS}) there exists a Banach space $Z$ with a $1$-convex block homogeneous basis $(z_i)$ so that the mapping $e_n\mapsto (u_n, z_n)$ defines an isomorphism from $X$ into $U\oplus Z$, where $U=[(u_n)_{n=1}^\infty]$. The basis $(z_n)$ is called the convex block homogeneous part of $(e_i)$ (see \cite[p. 1208]{AMS}). Now let $(x_n)$ be any non-trivial weak Cauchy sequence in $X$. From \cite[Theorem 7.1(i)]{AMS} it follows that $(x_n)$ has a convex block sequence $(\mathfrak{X}_n)$ which is either equivalent to the summing basis of $\co$ or equivalent to $(z_n)$. In the former case the conclusion follows similarly as in the proof of Proposition \ref{prop:1sec7}. Let us consider the latter case. Since $(\mathfrak{X}_n)\sim (z_n)$ and $(z_n)$ is $1$-convex block homogeneous, one can easily verify that $(\mathfrak{X}_n)$ is convex block shiftable. Moreover, one can easily verify that $(x_n - \mathfrak{X}_n)$ is weakly null and this completes the proof of theorem. 
\end{proof}

\begin{rem}\label{rem:1su} Notice that in the above proof we can also conclude that $(\mathfrak{X}_n)$ is uniformly equivalent to all of its subsequences. 
\end{rem}

\begin{rem} It follows that James's space $\mathrm{J}_p$ has property $(\mathfrak{su})$. 
\end{rem}

The next result provides a local description of weak compactness in spaces with property $(\mathfrak{su})$. 

\begin{lem}\label{lem:1sec7} Let $X$ be a Banach space with property $(\mathfrak{su})$ and $C\in \mathcal{B}(X)$. Then either $C$ is weakly compact, $C$ contains an $\ell_1$-sequence or $C$ contains a uniformly shift equivalent wide-$(s)$ sequence. 
\end{lem}

\begin{proof} Assume that $C$ is weakly compact. We already know that $C$ cannot contain an $\ell_1$-basic sequence. That $C$ does not contain a wide-$(s)$ sequence is e.g. a consequence of \cite[Proposition 2]{Ro1}. Assume that $C$ is not weakly compact. If $C$ contains some basic sequence equivalent to the unit basis of $\ell_1$, the result follows. Assume then that $C$ contains no $\ell_1$-sequences. Then $C$ contains a wide-$(s)$ sequence $(x_n)$ which is weak-Cauchy and has no weak convergent subsequences. So, $(x_n)$ is non-trivial weak Cauchy basic and there is $A>0$ so that $A \big| \sum_{i=1}^m a_i\big| \leq \big\| \sum_{i=1}^m a_i x_i\big\|$ for all scalars $(a_i)_{i=1}^m$. Since $X$ has property $(\mathfrak{su})$ there is a convex block shiftable  sequence $(w_n)$ in $X$ so that $(x_n - w_n)$ is weakly null. Passing to a subsequence if necessary, we may assume without loss of generality that all convex block sequences of $(w_n)$ are uniformly $L$-shift equivalent, for some $L>0$. By Mazur's Theorem we can find a convex block sequence $(\mathfrak{X}_n)$ of $(x_n)$ and a convex block sequence $(\mathfrak{W}_n)$ of $(w_n)$ so that $\lim_n\| \mathfrak{X}_n - \mathfrak{W}_n\| =0$. Note that $A\leq \inf_n\|\mathfrak{X}_n\|$. On the other hand, we know that $(\mathfrak{W}_i)$ satisfies
\[
\frac{1}{L}\Bigg\| \sum_{i=1}^m a_i \mathfrak{W}_i\Bigg\| \leq \Bigg\| \sum_{i=1}^m a_i \mathfrak{W}_{i+ p} \Bigg\| \leq L \Bigg\| \sum_{i=1}^m a_i \mathfrak{W}_i\Bigg\|,
\]
for all $p\in \mathbb{N}$ and scalars $(a_i)_{i=1}^m\subset \mathbb{R}$. It follows from Definition \ref{def:cbss} that this inequality passes to subsequences, i.e. for any increasing sequence $(n_i)\subset \mathbb{N}$ we also have $(\mathfrak{W}_{n_{i+p}})\sim_L (\mathfrak{W}_{n_i})$ for all $p\in \mathbb{N}$. Thus after taking another subsequence we may without loss of generality assume that 
\[
\sum_{i=1}^\infty \big\| \mathfrak{X}_i - \mathfrak{W}_i\big\| \leq \min\Big(\frac{ A}{ 4\mathscr{K}(1+L)}, \frac{ AL}{4\mathscr{K}(2+L)}\Big)
\]
where $\mathscr{K}$ denote the basis constant of $(x_n)$. Using these facts in concert with standard basic sequence techniques, one can prove 
\[
\frac{L}{2}\Bigg\| \sum_{i=1}^m a_i \mathfrak{X}_i\| \leq \Bigg\| \sum_{i=1}^m a_i \mathfrak{X}_{i+ p}\Bigg\| \leq L\Bigg\| \sum_{i=1}^m a_i \mathfrak{X}_i\|,
\]
for all $p\in \mathbb{N}$ and scalars $(a_i)_{i=1}^m\subset \mathbb{R}$.
\end{proof}

As we shall see next, in the above lemma, the assumption that $X$ has property $(\mathfrak{su})$ is needed in the proof, so in some sense it is sharp. Let $K$ be a compact Hausdorff space. For a given $K$, let $C(K)$ denote the Banach space of all continuous real-valued functions $f\colon K\to \mathbb{R}$, equipped with its usual sup-norm. Recall that $K$ is said to be scattered if every closed subset $L\subset K$ has an isolated point.  Moreover, $K$ is scattered if and only if there is an ordinal $\alpha$, such that the Cantor-Bendixon derivative $K^{(\alpha)}$ is empty. The {\it height} of a scattered compact space is the smallest ordinal $\alpha$ such that $K^{(\alpha)}$ is empty. When such ordinal exists for a space $K$, we say that $K$ has finite height, otherwise $K$ has infinite height. The cardinal number $w(K)=\min \{ \# \mathcal{B}\colon \mathcal{B} \text{ is a basis for } $K$\}$ is called the weight of $K$. The space $K$ is called {\it Eberlein compact} if it is homeomorphic to a weakly compact subset of a Banach space. Finally, recall that for a set $\Gamma$, $\co(\Gamma)$ is the Banach space of all scalar-valued maps $f$ on $\Gamma$ with the property that for every $\varepsilon>0$ the set $\{ \gamma \in \Gamma\colon | f(\gamma)| \geq \varepsilon\}$ is finite, equipped with the sup-norm. 

\smallskip 
\begin{prop}\label{prop:3sec7} Let $K$ be an infinite compact Hausdorff space. Then the following statements hold. 
\begin{itemize} 
\item[(i)] If $K$ is not scattered then there exists a set $C\in \mathcal{B}(C(K))$ that is neither weakly compact, does not contain $\ell_1$-sequences nor any shift equivalent sequences. So, $C(K)$ fails to have property $(\mathfrak{su})$. 
\item[(ii)] If $K$ is Eberlein compact with weight $<\omega_\omega$, then $C(K)$ has property $(\mathfrak{su})$ if and only if $K$ has finite height. 
\end{itemize}
\end{prop}

\begin{proof} (i) If $K$ is not scattered, then $C[0,1]$ isomorphically embeds into $C(K)$ (cf. \cite[p. 629(v)]{FHHMZ}). It follows from Banach-Mazur's embedding theorem that every separable Banach space embeds into $C(K)$. Let $X_{GM}$ denote the space constructed by Gower and Maurey \cite{GM}. Accordingly, $X_{GM}$ has the property that it contains no unconditional basic sequence and is hereditarily indecomposable. Thus no subspace of $X_{GM}$ can be isomorphic to any proper subspace. It follows that the range of the unit ball $B(X_{GM})$ in $C(K)$ fulfills the properties stated in assertion (i). The remaining property follows directly from Lemma \ref{lem:1sec7}. 

(ii) If $K$ is not of finite height, then it is not scattered and by (i) the result follows. Now assume that $K$ is of finite height. Since $K$ is Eberlein compact with weight $<\omega_\omega$, a result of Godefroy, Kalton and Lancien \cite[Theorem 4.8]{GKL} ensures that $C(K)$ is linearly isomorphic to some $\co(\Gamma)$. On the other hand it is known that $\co(\Gamma)$ is an $M$-ideal in its bidual \cite[Example III.1.4]{H}. Then, by a result of Godefroy and Li \cite[Theorem 1]{GL} the space $\co(\Gamma)$ has Pe\l czy\'nski's property $(u)$. Since this property is invariant by isomorphisms, $C(K)$ has property $(u)$ as well. By Proposition \ref{prop:1sec7} the result follows, and the proof is concluded.  
\end{proof}

\smallskip 
The proof of assertion (i), in turn, immediately yields a corollary.

\begin{cor}\label{cor:1sec7} Every Banach space $X$ containing an isomorphic copy of $X_{GM}$ contains a set $C\in \mathcal{B}(X)$ that neither is weakly compact, nor contains $\ell_1$ or any shift equivalent basic sequence. 
\end{cor}

We conclude this section with another consequence of Lemma \ref{lem:1sec7}, that provides a fixed point characterization of weak compactness in spaces with property $(\mathfrak{su})$. We note that this answer question raised in Remark \ref{rem:23}.

\begin{thm}\label{thm:M3} Let $X$ be a Banach space with property $(\mathfrak{su})$. Then $C\in \mathcal{B}(X)$ is weakly compact if and only if $C$ has the $\mathcal{G}$-FPP for the class of uniformly bi-Lipschitz affine mappings $f$ such that $\theta(f)=0$. 
\end{thm}

\begin{proof} The proof follows the same argument as that in the proof of Theorem \ref{thm:6.1}, except that Lemma \ref{lem:KL2} must be replaced by Lemma \ref{lem:1sec7}. 
\end{proof}



\smallskip 
\section{Concluding remarks and acknowledgements}

We can summarize our fixed point results as follows:
\begin{thm} Let $C$ be a closed convex bounded set in a Banach space $X$. 
\begin{itemize}
\item[(a)] If $C$ is not weakly compact, then for every $L>1$ there are a closed convex subset $K\subset C$ and an affine $L$-bi-Lipschitz mapping $T\colon K\to K$ such that $\theta(T)>0$. \vskip .1cm 
\item[(b)] If $X$ has property $(u)$, then either $C$ is weakly compact, contains a $\ell_1$-sequence or contains a $\co$-summing basic sequence. If $C$ is not weakly compact, then there are a closed convex subset $K\subset C$ and an affine uniformly bi-Lipschitzian mapping $T\colon K\to K$ such that $\theta(T)>0$. \vskip .1cm 
\item[(c)] If $X$ has property $(\mathfrak{su})$ then either $C$ is weakly compact, contains an $\ell_1$-sequence or contains a uniformly shift equivalent wide-$(s)$ sequence. If $C$ is not weakly compact, then there are a closed convex subset $K\subset C$ and an affine uniformly bi-Lipschitzian mapping $T\colon K\to K$ such that $\theta(T)>0$. 
\end{itemize}
\end{thm}

It is worthy of remark that this result yields a significant generalization of Theorem 3.2 in \cite{BPP}. Assertion (a) in \cite[Thm. 3.2]{BPP} works only for continuous maps. Here, assertions (b) and (c) encompass a more wider class of spaces than corresponding assertions in \cite[Thm. 3.2]{BPP}. 

In \cite[Theorems 3.4 and 4.2]{BP2}, Benavides and Jap\'on-Pineda have characterized weak compactness in spaces with $1$-unconditional basis being either boundedly complete or shrinking,  in terms of the FPP for cascading nonexpansive mappings. The class of cascading nonexpansive mappings was introduced by Lennard and Nezir \cite{LN2}. It would be important to know under which conditions statement (b) or (c) in the above result could be improved in order to restrict the class of maps to the class of cascading nonexpansive (affine) mappings, respectively, in the context of properties $(u)$ or $(\mathfrak{su})$. 

Let us point out that in general the restriction on the Eberlein compactness of $K$ in Proposition \ref{prop:3sec7} cannot be improved. Indeed, let $\beta\mathbb{N}$ be the Stone-\v{C}ech compactification of the natural numbers. Then $\beta\mathbb{N}$ is scattered but not angelic. Hence it cannot be Eberlein compact. Moreover, since $\ell_\infty \equiv C(\beta\mathbb{N})$ and every separable Banach space isometrically embeds into $\ell_\infty$, by Corollary \ref{cor:1sec7} the space $C(\beta\mathbb{N})$ fails property $(\mathfrak{su})$.

Very recently, Freeman, Odell, Sari and Zheng in \cite{FOSZ} have proved that every Banach space with a spreading basis contains a complemented subspace with an unconditional basis. Notice however that Lemma \ref{lem:1sec7} is no longer true if $X$ is merely assumed to contain a complemented subspace with an unconditional basis. Indeed, for any infinite dimensional Banach space $E$, by a result of Cembranos \cite{Cem} we know  that $\co$ is always complemented in the Banach space $C([0,1]; E)$ of all continuous $E$-valued functions defined on $[0,1]$ and equipped with the supremum norm. Therefore, since $C([0,1])$ isometrically embeds into $C([0,1];E)$, by Corollary \ref{cor:1sec7} there is a set $C\in \mathcal{B}( C([0,1]; E)$ that neither is weakly compact, nor contains $\ell_1$-sequences or any shift equivalent basic sequence. 

\medskip 
\subsection*{Acknowledgements}\label{ackref}
This work was started while the first author was a Visiting Researcher Scholar at Texas A$\&$M University (2015--2016) under the supervision of Professor Thomas Schlumprecht. The first author would also like to warmly thank Professor Bill Johnson for many didactical conversations on several technicalities related to this work. The authors would also like to thank the anonymous referees for their helpful comments that led to the improved version of the manuscript. Part of the material presented in Section 7 was first announced in the Joint Mathematics Meeting 2019 of the American Mathematical Society. The first named author also wish to thank Professors Chris Lennard and Torrey Gallagher for the kind invitation to give a talk about the main topic of Section 7 in JMM-2019.



\end{document}